\documentclass[14pt]{amsart}

\usepackage{amssymb}
\usepackage{amstext}
\usepackage{amsmath}
\usepackage{amscd}
\usepackage{latexsym}
\usepackage{amsfonts}
\usepackage{color}
\usepackage{enumerate}
\usepackage[all]{xy}
\usepackage{graphicx}

\theoremstyle{plain}
\newtheorem{thm}{Theorem}[section]
\newtheorem*{thm*}{Theorem}
\newtheorem*{cor*}{Corollary}
\newtheorem*{defn*}{Definition}

\newtheorem{lem}[thm]{Lemma}

\newtheorem*{claim*}{Claim}

\usepackage[pagebackref]{hyperref}
\usepackage{hyperref}
\hypersetup{
    colorlinks=true, 
    linkcolor=red, 
}

\usepackage[alphabetic,initials]{amsrefs}

\theoremstyle{definition}
\newtheorem{defn}[thm]{Definition}

\newtheorem{rem}[thm]{Remark}

\newtheorem{lm}[thm]{Lemma}

\theoremstyle{remark}

\begin{document}

\title{The affine cones over Fano-Mukai fourfold of genus $7$ are flexible }

\author{Nguyen Thi Anh Hang}
\email{hangnthianh@gmail.com}
\address{Thai Nguyen University of Education, Thai Nguyen City, Viet Nam}

\author{Hoang Le Truong}
\email{hltruong@math.ac.vn\\
	truonghoangle@gmail.com}
\address{Institute of Mathematics, VAST, 18 Hoang Quoc Viet Road, 10307
Hanoi, Viet Nam}

\thanks{
N. T. A. Hang was partially supported by Grant number ICRTM02 2021.04, awarded in
the internal grant competition of the International Center for Research and Postgraduate
Training in Mathematics, Hanoi and IMU Breakout Graduate Fellowship (IMU-BGF-2021-
03). H. L. Truong was partially supported by a fund of VAST under grant number NVCC01.04/22-23.}

\begin{abstract}
In this paper, we will show that the affine cones over any smooth Fano-Mukai fourfold of genus $7$ are flexible.
\end{abstract}

\maketitle

%\vspace*{6pt}\tableofcontents  % for this guide only.
% A table of contents should normally not be included
\section{Introduction}

 All considered varieties are algebraic and defined over the complex  field $\Bbb{C}$. 
 
The study of  rational Fano varieties is an active area of research in algebraic geometry, as these varieties have many interesting properties and are related to other areas of mathematics and physics. In this paper, we study rational Fano varieties with flexible affine cones.
The concept of flexibility of affine varieties has been studied in an algebraic context in \cite{AZK12}. The interest in flexible varieties is significant, not only because of their remarkable algebraic and geometric consequences, but also due to the surprising breadth of the class of such varieties (see \cite{Arz23}). Moreover, it has been shown that the flexibility of affine varieties has several useful properties, such as unirationality and has been applied to problems in algebraic geometry such as the Zariski cancellation problem. Additionally, the flexibility of affine varieties is closely related to the existence of a transitive additive group action on the affine cone over a smooth projective variety, as shown in \cite{AZK12, AFKKZ13}.
  
  Now, let $X$ be a smooth projective variety and $H$ an arbitrary ample divisor on $X$, the affine cone over $X$ is an affine variety defined as follows. 
    $${\rm Affcone}_H(X) = {\rm Spec} \bigoplus_{m=0}^\infty H^0(X, \mathcal{O}_X(mH)).$$
 \noindent  
It is known that if the affine cone over a smooth projective variety with Picard rank $1$ is flexible, then the variety is a rational Fano  variety, but it is not clear if the converse is true, whether every rational Fano  variety with Picard rank $1$ has a flexible affine cone.

 In dimension $2$, 
 the flexibility of affine cones over del Pezzo surfaces of degree $4$ is well-known (see \cite{Per13}). In \cite{PaW16}, the authors extended this result to show that for every ample divisor, the affine cones over del Pezzo surfaces of degree $4$ are flexible, and this is also true for affine cones over del Pezzo surfaces of degree at least $5$.
In dimension $3$, it has been shown that the affine cones over certain specific Fano threefolds in Mori-Mukai's classification are flexible (see \cite{MPS18}). In dimension $4$, the flexibility of the affine cone over {\it every} Fano-Mukai fourfold of genus $10$ has been proved in \cite{PrZ20}. More recently, Hoff and the last author showed in \cite{HT22} that the same is true for the affine cones over {\it general} Fano-Mukai fourfolds of genus $7$, $8$, and $9$. Therefore, a natural question arises: are the affine cones over {\it every} Fano-Mukai fourfold of genus $7$, $8$, and $9$ flexible? (see  \cite[Problem 5.11]{HT22}). In this paper, we provide a positive answer to this question for the case of genus $7$ (see Theorem \ref{mainthm}).

\subsection{Notation}
 Let $Y\subset\Bbb P^N$ be a smooth projective variety of dimension $n$ and canonical class $K_Y$. We denote by $H_Y$ the hyperplane class of $Y$ and $d_Y$ the degree of $Y$. The sectional genus $\pi_Y$ of $Y$ is defined as follows. $$\pi_Y=\frac{1}{2}H_Y^{n-1}((n-1)H_Y+K_Y)+1.$$ The tangent bundle and the  $i$-th Chern class of a vector bundle $\mathcal E$ on $Y$ are denoted by $\mathcal{T}_Y$ and $c_i(\mathcal E)$, respectively. Moreover, $c_i(Y) = c_i(\mathcal T_Y)$. The Euler-Poincar\'{e} characteristic $\chi_Y $ of $Y$ is defined by $\chi_Y = \chi(\mathcal O_Y).$ 

The paper is divided into four parts. The main results are presented in Section \ref{main}, where we prove the geometric properties of Fano-Mukai fourfolds of genus $7$. In Section \ref{birational}, we describe the birational maps on these fourfolds. The last section is dedicated to prove the main theorem of the paper.

\section{Main results} \label{main}

In this paper, we study the geometric properties of the intersection of a Fano-Mukai fourfold of genus $7$ denoted as $X$ with the tangent space of $X$ at a point $p$. Previous results from \cite{LM03} and \cite{Kuz18} show that this intersection can take one of the following forms: (i) a set of $5$ lines, (ii) containing a plane or a quadratic surface, or a cubic scroll surface. Our aim is to describe birational maps from $X$ to certain Fano varieties based on these geometric properties. Note that case (i) was already studied in \cite{HT22}, and case (ii) was studied in \cite{HHT22} under the assumption that the surfaces are smooth cubic scroll surfaces, but the other cases are not known. In particular, we have the following theorems.

\begin{thm}\label{pro1}
	Let $ X \subseteq \Bbb P^{9}$ be a Mukai fourfold of genus $7$ with ${\mathrm{Pic}}(X ) \cong \Bbb Z L$ for the hyperplane class $L$. 	Suppose that $X$ contains   a  surface $\Sigma$  as in Table \ref{S95}. Then the following statements hold true.
	\begin{enumerate}[$(1)$]
		\item The linear system $|L - \Sigma|$  passing through $\Sigma$ defines a birational map $\Phi:X \dashrightarrow  Y$.
		
		\item There is a commutative diagram
		$$\xymatrix{&D\ar@{^{(}->}[r]\ar[ld]&\widetilde{X}\ar[dl]_{\varphi}\ar[dr]^{\rho}&E\ar@{_{(}->}[l]\ar[rd]&\\
			\Sigma\ar@{^{(}->}[r]&X\ar@{-->}[rr]^{\Phi}&&Y&F,\ar@{_{(}->}[l]}$$
		where $\varphi$ is the blowup of $X$ centered at the surface $\Sigma$ with exceptional divisor $D$ and $\rho$  is a birational morphism defined by the linear system $|\varphi^\ast L - D|$.
		
		\item The $\rho$-exceptional locus is an irreducible divisor $E\subset \widetilde{X}$.
		\item Let  $H$ be the ample generator of ${\mathrm{Pic}}(Y)$. Then
		$$\begin{aligned}
			\varphi^\ast L \sim \,(d_{\Sigma}+1)\rho^\ast H- E,\quad\quad\quad\quad&\quad\quad D\sim d_{\Sigma}\,\rho^\ast H-E,\\
			\rho^\ast H \sim \varphi^\ast L- D, \quad\quad\quad\quad&\quad\quad E\sim d_{\Sigma}\,\varphi^\ast L-(d_{\Sigma}+1)D.
		\end{aligned}
		$$

		\item The image $\rho(E)$ is  a surface $F\subset Y$ with at most isolated singularities. The birational map $\rho$ is the blowup centered at $F$. Moreover, if $\Sigma$ is smooth then $F$ and $Y$ are smooth. 
		\item $\rho(D)$ is a unique hypersurface of degree $d_{\Sigma}$ containing $F$  and $\varphi(E)$ is a hypersurface of degree $d_{\Sigma}$ of  $X$ containing $\Sigma$.
		
		\item $X\backslash \varphi(E)\cong Y\backslash \rho(D)$.
		
	\end{enumerate}
	
\end{thm}

\begin{table}[h!]
	\centering
	\begin{tabular}{ |r| c| c|| c|c|c|c|c|c|}
		\hline
		$\Sigma\subset X$&$d_{\Sigma}$ &  $\pi_\Sigma$ &$Y$&$d_F$&$\pi_F$&$(\varphi^\ast L)^2\cdot D^2$&$(\varphi^\ast L)\cdot D^3$&$D^4$\\ 
		\hline
		plane &$1$   & $0$   &$X_{2\cdot2}$&$5$ & $1$ &$-1$ &$1$ &$2$ 	\\ 
		\hline
		quadric surface &$2$   & $0$   &$Q^4\subseteq\Bbb P^5$&$8$ & $4$ &$-2$ &$0$ &$2$ 	\\ 
		\hline
		cubic scroll surface &$3$   & $0$   &$\Bbb P^4$&$7$ & $3$ &$-3$ &$-1$ &$3$ 	\\ 
		\hline
		
	\end{tabular}
	\caption{ Surfaces $\Sigma$ contained in a Mukai fourfold $X$ of genus $7$ and surfaces $F\subset Y$.}
	\label{S95}
\end{table}

The inverse picture of  the Theorem \ref{pro1} would be the inverse of the birational map $\Phi$, describing a birational map from the Fano variety back to the Fano-Mukai fourfold of genus $7$. The exact details of the inverse picture would be provided in the following theorem.

\begin{thm}\label{invpro}
	Let $F$ and $Y$ be the varieties  as in Table \ref{S95}. Then the following statements hold true.
	\begin{enumerate}[$(1)$]
		\item The linear system $|(i-1)H - F|$ passing through $F$ defines a birational map $\Phi:Y \dashrightarrow X \subset \Bbb P^{9}$, where $X =\Phi(Y)$ is a Mukai fourfold of genus $7$ with ${\mathrm{Pic}}(X ) \cong \Bbb Z L$ for an ample generator $L$, where  $i$ is  the index of $Y$.
		
		\item There is a commutative diagram
		$$\xymatrix{&E\ar@{^{(}->}[r]\ar[ld]&\widetilde{Y}\ar[dl]_{\rho}\ar[dr]^{\varphi}&D\ar@{_{(}->}[l]\ar[rd]&\\
			F\ar@{^{(}->}[r]&Y\ar@{-->}[rr]^{\Phi}&&X&\Sigma,\ar@{_{(}->}[l]}$$
		where $\rho$ is the blowup of $Y$ centered at $F$ with exceptional divisor $E$  and  $\varphi$  is a birational morphism defined by the linear system $|(i-1)\rho^\ast H - E|$ with $\varphi$-exceptional divisor $D$.

		\item $\rho(D)$  is a singular cubic threefold. The singularities of $\Sigma$ are the $\Phi$-images of planes in $\rho(D)$ meeting $F$ along a quartic curve. The birational map $\varphi$ is the blowup of a surface $\Sigma$ as in Table \ref{S95}. Moreover, with a general  surface $F$, the varieties $\Sigma$ and $X$ are smooth.
		
		\item We have the following relations 
		$$\begin{aligned}
			\varphi^\ast L \sim (i-1)\rho^\ast H- E,\quad\quad\quad\quad&\quad\quad D\sim (i-2)\rho^\ast H-E,\\
			\rho^\ast H \sim \varphi^\ast L- D, \quad\quad\quad\quad&\quad\quad E\sim (i-2)\varphi^\ast L-(i-1)D.
		\end{aligned}
		$$
	\end{enumerate}
		
\end{thm}
Using the above theorems, we give an affirmative answer to the question in \cite[Problem 5.11]{HT22} in the case of genus $7$. In particular, we have the following theorem.
\begin{thm}\label{mainthm}
	The affine cone over any smooth Fano-Mukai fourfold of genus $7$ is flexible.
\end{thm}

\begin{rem}\rm	
To establish the above results, we examine the intersection of a Fano-Mukai fourfold $X$ of genus $7$ with its tangent space at every point on $X$, along with the tangent cone at each point. While in \cite{HT22}, M. Hoff and the last author investigated a scenario where the intersection occurs at a general point on a Fano-Mukai fourfold $X$ of genus $7$, we concentrate on the remaining cases in this paper.  
The proof utilizes criteria for the flexibility of affine cones from \cite[Theorem 5]{Per13} and \cite[Lemma 5.7]{HT22}, which are based on certain open coverings of the projective variety. Starting with a Fano-Mukai fourfold $X$ of genus $7$, we demonstrate the existence of an open subset $U$ of $X$ that can be covered by $\mathbb{A}^2$-cylinders. We then establish the transversality of this covering using the aforementioned flexibility criteria.

\end{rem}

%%%%%%%%%%%%%%%%%%%%%%%%%%%%%%%%%%%%%%%%%%%%%%%%%%%%
%%%%%%%%%%%%%%%%%%%%%%%%%%%%%%%%%%%%%%%%%%%%%%%%%%%%
%%%%%%%%%%%%%%%%%%%%%%%%%%%%%%%%%%%%%%%%%%%%%%%%%%%%
%%%%%%%%%%%%%%%%%%%%%%%%%%%%%%%%%%%%%%%%%%%%%%%%%%%%

\section{Birational maps on Fano-Mukai fourfolds of genus $7$} \label{birational}

\subsection{Birational maps from Fano-Mukai fourfolds of genus $7$}
In this section, we will prove Theorem \ref{pro1}. Assume that $X$ is a Fano-Mukai fourfold of genus $7$ with  the hyperplane class $L$ containing a surface $ \Sigma$ as describe in Table \ref{S95}. We denote by $\varphi: \widetilde{X} \longrightarrow X $ the blow up of $X$ centered at the surface $\Sigma$ with exceptional divisor $D$. 
\begin{lm}{\it \label{fano} The variety $\widetilde{X}$ is a Fano fourfold with Picard rank $2$.}
\end{lm} 

\begin{proof} 
 Since  ${\rm Pic}\,(X) = \mathbb{Z}L$ and $\varphi $ is the blowup of $X$ centered at $\Sigma$, we have $\textrm{rank }  {\rm Pic}\,(\widetilde{X})=2$ and
$$-K_{\widetilde{X}} = \varphi^\ast L+\varphi^\ast L - D.$$  Because the projection from the $(7-d_{\Sigma})$-plane  $\langle\Sigma\rangle$ induces a birational map $\Phi:X \dashrightarrow  Y$, the linear system $|\varphi^\ast L - D|$ is base point free. Thus $\varphi^\ast L - D$ is nef. Since $ \varphi^\ast L$ is ample, by the Kleiman's criterion for ampleness, their sum is an ample divisor. Hence $\widetilde{X}$ is a Fano fourfold with Picard rank $2$, as required. 
\end{proof}

\begin{lm}{\it \label{mori} There exist two different Mori contractions on $\widetilde{X}.$}
\end{lm}
\begin{proof} 
	By the Cone theorem, the two nef linear systems $|\varphi^\ast L - D| $ and $|\varphi^\ast L|$ define two Mori contractions on $\widetilde{X}.$ The first one is defined by $|\varphi^\ast L|$ which is exactly $\varphi: \widetilde{X} \longrightarrow X.$ The second one is $\rho: \widetilde{X} \longrightarrow U$  defined by $|\varphi^\ast L - D| $ that is different from $\varphi.$ 
\end{proof}

\begin{lem}\label{internumber}\label{rem1} We have
	\begin{enumerate}[$(1)$]
	 \item  $(\varphi^\ast L)^4=12,$ $(\varphi^\ast L)^3\cdot D=0,$ and $(\varphi^\ast L)^2\cdot D^2$, $(\varphi^\ast L)\cdot D^3$ and $D^4$  are as in Table \ref{S95}. 
	 \item The divisor class of $\varphi^\ast L-D$ is not ample.	
\end{enumerate}
\end{lem}
\begin{proof}

$(1)$.  By \cite[Lemma 2.3]{PrZ16}, we have $(\varphi^\ast L)^4=L^4.$ It is exactly the degree of $X$. Hence $(\varphi^\ast L)^4 = 12$. We also have $(\varphi^\ast L)^3\cdot D=0$. Moreover, the intersection numbers $(\varphi^\ast L)^2\cdot D^2$ $(\varphi^\ast L)\cdot D^3$ and $D^4$ are as in Table \ref{S95}. \\
$(2)$. Since $(\varphi^\ast L-D)^3\cdot (d_{\Sigma}\,\varphi^\ast L-(d_{\Sigma}+1)D)=0$, by the Kleiman's criterion for ampleness, the divisor class of $\varphi^\ast L-D$ is not ample.
	\end{proof}

Recall that a contraction $\phi: X \longrightarrow Y$ is divisorial if ${\rm dim}\, X = {\rm dim}\, Y$ and the codimension of the $\phi$-exceptional locus is one. In the next lemma,  we will show that the Mori contraction $\rho$ is divisorial.

\begin{lem} \label{divisorial} The Mori contraction $\rho$ is divisorial. Moreover, the $\rho$-exceptional locus is an irreducible divisor $E\subseteq \widetilde{X}.$
\end{lem}

\begin{proof} Since the Mori contraction $\rho: \widetilde{X} \longrightarrow U$ is determined by the nef linear system $|\varphi^\ast L-D|,$  by the Contraction theorem, we can write $\varphi^\ast L-D = \rho^\ast H$ where $H$ is the ample divisor generated ${\rm Pic}\,(U)$. By the Riemann-Roch and the Kodaira Vanishing theorem, we have ${\rm dim}\,|\varphi^\ast L-D|=4={\rm dim}\, X.$ Hence $\rho$ is a birational map. It follows that ${\rm dim}\, U = {\rm dim}\,X =4.$ From Lemma \ref{rem1}, we get $(\varphi^\ast L-D)^3\cdot (d_{\Sigma}\,\varphi^\ast L-(d_{\Sigma}+1)D)=0$. By the Contraction theorem, there exists a divisor $E$ in U contracted by $\rho$, i.e. the codimension of the $\rho$-exceptional locus is one.  Thus, $\rho$ is a divisorial contraction.
	 Moreover, by the properties of Contraction of divisorial type (see \cite[Proposition 8-2-1]{Mat02}), the $\rho$-exceptional locus contains a unique irreducible divisor $E$. \end{proof}

	\begin{lem}\label{sing.F}
	The image $\rho(E)$ is  a surface $F\subset U$ with at most isolated singularities.
	\end{lem}

\begin{proof}
Since $(\varphi^\ast L-D)^2\cdot E^2 = -d_F$, the image $F=\rho(E)$ is a {\it surface} with $\deg F = H^2\cdot \rho(E) = d_F$ (as in Table \ref{S95}). 
Since ${\mathrm{rank}}\ {\mathrm{Pic}}\, (\widetilde{X})= 2$, the exceptional locus of $\rho$ coincides with $E$ and so $E$ is a prime divisor. Therefore, by \cite[Theorem 4.1.3]{AnM03}, $\rho$ has at most a finite number of $2$-dimensional fibers. Hence, by the Theorem in \cite{AnW98}, $F$ has at most isolated singularities. 
\end{proof}
	
	\begin{lem}\label{rhoblowup}
	 The birational map $\rho$ is the blowup of  $U$ belong the surface  $F$ with the exceptional divisor $E$. 
	\end{lem}
\begin{proof} By Lemma \ref{divisorial}, $\rho$ is a contraction of divisorial type and its exceptional locus coincides with $E$. Thus ${\rm dim} \,E = 3$. By Lemma \ref{sing.F}, we have ${\rm dim}\, \rho(E)=2$. Moreover, $-K_{\widetilde{X}}$ is ample by Lemma \ref{fano}. Therefore, by \cite[Theorem 4.1.3]{AnM03}, we conclude that $\rho$ is the blowup of  $U$ belong the surface  $F$ with the exceptional divisor $E$. 
\end{proof}	

\begin{lem}\label{relation} We have the following relations
	$$\begin{aligned}
		\varphi^\ast L \sim (d_{\Sigma}+1)\rho^\ast H- E,\quad\quad\quad\quad&\quad\quad D\sim d_{\Sigma}\,\rho^\ast H-E,\\
		\rho^\ast H \sim \varphi^\ast L- D, \quad\quad\quad\quad&\quad\quad E\sim d_{\Sigma}\,\varphi^\ast L-(d_{\Sigma}+1)D.
	\end{aligned}
	$$
\end{lem}
\begin{proof}The relation $\rho^\ast H \sim \varphi^\ast L- D$ follows from Lemma \ref{divisorial}. By Lemma \ref{rem1}, we have
	 $$(\varphi^\ast L-D)^3\cdot (d_{\Sigma}\,\varphi^\ast L-(d_{\Sigma}+1)D)=0.$$ Thus it follows from \cite[Lemma 2.3]{PrZ16} and Lemma \ref{divisorial} that $E \sim d_{\Sigma}\,\varphi^\ast L-(d_{\Sigma}+1)D.$ The two remain relations follows.
\end{proof}

\begin{lem} \label{var.Y}
 The variety $U$  is the Fano variety $Y$ as in Table \ref{S95}.
\end{lem}
\begin{proof}
	It follows from the birationality of $\rho$ and the ampleness of $-K_{\widetilde{X}}$ that the variety $U$ is a Fano variety. Moreover, we have ${\rm dim}\, U = {\rm dim}\, |\rho^\ast H| = {\rm dim} \,|\varphi^\ast L -D|=4.$ By the above discussion, rank of ${\rm Pic}\, (U)$ is $1$. Since $H$ is the ample generator of ${\rm Pic}\, (U)$, we get ${\rm d}_U=H^4$. From Lemma \ref{relation} we have
	$$-K_{\widetilde{X}}=2\varphi^{\ast}L - D = (d_{\Sigma} + 2) \rho^{\ast}H - E,$$
	 By Lemma \ref{rhoblowup} we have 	$-K_U = (d_{\Sigma}+2)H.$ Thus the index of $U$ is $d_{\Sigma}+2$. By the classification of Fano fourfolds with index $\ge 2$, the variety $U$ is exactly the Fano fourfold $Y$ as in Table \ref{S95}.	
\end{proof}

From the above lemmas, we can proceed the proof of Theorem \ref{pro1} as follows.
\begin{proof}[Proof of Theorem \ref{pro1}] 
To $(1)$ and $(2)$: It follows from Lemma \ref{fano} that $\widetilde{X}$ is a Fano fourfold with Picard rank $2$. By Lemma \ref{mori}, there is a contraction $\rho: \widetilde{X} \longrightarrow U$ different from $\varphi.$ It implies from Lemma \ref{rem1} that $\varphi^\ast L-D$ is not ample, so it yields a supporting linear function of the extremal ray generated by the curves in the fibers of $\varphi$. By Lemma \ref{divisorial}, $\rho$ is a birational map.  The birationality of $\Phi$ and the commutativity of the diagram follow. By Lemma \ref{var.Y} the variety $U$ is identified with the variety $Y$ as in Table \ref{S95}.

To $(3)$: It follows from Lemma \ref{divisorial}.

To $(4)$: It is immediately from Lemma \ref{relation}. 

To $(5)$: It follows from Lemma \ref{sing.F} and Lemma \ref{rhoblowup}. Moreover, if $\Sigma$ is smooth then $F$ and $Y$ are smooth. 

To $(6)$: From the above discussion, we have $E=d_{\Sigma}\,\varphi^\ast L-(d_{\Sigma}+1)D$ in ${\widetilde{X}}$ and $\Sigma \subset \varphi(E)$, because for any  $s \in \Sigma$, the fiber $\varphi^{-1}(s)$ meets $E$. Thus, $\varphi(E) \cong L$ is a hyperplane section of $X \subset \Bbb P^{9}$ singular along $\Sigma = \varphi(D)$.

To $(7)$: Since $F \subset \rho(D)$ we have isomorphisms
$$X\backslash \varphi(E)\cong {\widetilde{X}}\backslash(E\cup D)\cong {Y}\backslash(F\cup \rho(D)) \cong Y\backslash \rho(D).$$
\end{proof}

%%%%%%%%%%%%%%%%%%%%%%%%%%%%%%%%%%%%%%%%%%%%%%%%%%%%
%%%%%%%%%%%%%%%%%%%%%%%%%%%%%%%%%%%%%%%%%%%%%%%%%%%%
%%%%%%%%%%%%%%%%%%%%%%%%%%%%%%%%%%%%%%%%%%%%%%%%%%%%
%%%%%%%%%%%%%%%%%%%%%%%%%%%%%%%%%%%%%%%%%%%%%%%%%%%%
%%%%%%%%%%%%%%%%%%%%%%%%%%%%%%%%%%%%%%%%%%%%%%%%%%%%
%%%%%%%%%%%%%%%%%%%%%%%%%%%%%%%%%%%%%%%%%%%%%%%%%%%%

%%%%%%%%%%%%%%%%%%%%%%%%%%%%%%%%%%%%%%%%%%%%%%%%%%
%%%%%%%%%%%%%%%%%%%%%%%%%%%%%%%%%%%%%%%%%%%%%%%%%%
%%%%%%%%%%%%%%%%%%%%%%%%%%%%%%%%%%%%%%%%%%%%%%%%%%
%%%%%%%%%%%%%%%%%%%%%%%%%%%%%%%%%%%%%%%%%%%%%%%%%%

\subsection{Birational maps to Mukai fourfolds of genus $7$}
 In this subsection, we will prove Theorem \ref{invpro}. Assume that $F$ and $Y$ are the varieties  as in Table \ref{S95}. We denote by $i$ the index of $Y$ and $\rho: \widetilde{Y} \longrightarrow Y $ the blow up of $Y$ centered at the surface $F$ with exceptional divisor $E$.

\begin{lem} \label{fano2} The variety $\widetilde{Y}$ is a Fano fourfold with rank of $\textrm{\rm Pic} (\widetilde{Y})$ is two.
\end{lem}
\begin{proof} We have 
	$$-K_{\widetilde{Y}} = \rho^\ast H +(i-1)\rho^\ast H - E.$$
\noindent	 Since $F$ is a scheme-theoretic intersection, the linear system $|(i-1)\rho^\ast H - E|$ is base point free. Thus it is nef.  It follows from  $\rho^\ast H $ is ample and the Kleiman's criterion for ampleness that the divisor  $-K_{\widetilde{Y}}$  is ample, i.e. $\widetilde{Y}$ is a Fano fourfold. Moreover, rank of ${\mathrm{Pic}}\,(\widetilde{Y})$ is $2$. 
\end{proof}

\begin{lem}\label{mori2} There exist two different Mori contractions on $\widetilde{Y}.$
\end{lem}

\begin{proof} By the Cone theorem, the two nef linear systems $|(i-1)\rho^\ast H - E|$ and $|\rho^\ast H| $ defined two Mori contractions on $\widetilde{Y}.$ The first one is defined by $|\rho^\ast H|$ which is exactly $\rho: \widetilde{Y} \longrightarrow Y.$ The second one is $\mu: \widetilde{Y} \longrightarrow X'$  defined by $|(i-1)\rho^\ast H - E|$ that is different from $\rho.$ 
\end{proof}
\label{rem3} Let  $D$ be  the proper transform  of the unique hypersurface of degree $(i-2)$ of $Y$ that passes through $F$. Then we can write $D \sim (i-2)\rho^\ast H-kE$ for some $k > 0$.

\begin{lem}\label{noample2} The divisor class of $(i-1)\rho^\ast H-E$ is not ample.
\end{lem}
  \begin{proof}
	By \cite [Lemma 2.3]{PrZ16},  we have $$0\le ((i-1)\rho^\ast H-E)^3\cdot D=-12d_{\Sigma}\,(k-1).$$ Thus $k=1$ and $((i-1)\rho^\ast H-E)^3\cdot D=0$. 
	Hence $(i-1)\rho^\ast H-E$ is not ample.
	\end{proof}
	
	\begin{lem}\label{birational2} The Mori contraction $\mu$ is birational and its exceptional locus coincides with $D$.
	\end{lem}
\begin{proof}
	By the Contraction theorem, we can write $(i-1)\rho^\ast H-E=\mu^\ast L$, where $L$ is the ample generator of ${\mathrm{Pic}}\,(X') \cong \Bbb Z$.  Then we have $((i-1)\rho^\ast H-E)^4 = 
	12$ and $\dim \mu( \widetilde{Y}) = 4$.  
	Thus, $\mu$ is birational. Moreover, $\mu$ is divisorial, hence its exceptional locus coincides with $D$. In particular, it is an irreducible divisor. 
	 \end{proof}
 
	\begin{rem}\label{rem3}

	Using the Riemann-Roch and Kodaira Vanishing Theorems, we obtain the equality $\dim |(i-1)\rho^\ast H-E| = 9$. 
	This yields the diagram
	$$\xymatrix{&\widetilde{Y}\ar[dr]^{\mu}\ar[dl]^{\rho}\ar[r]^{\varphi}&X\ar[d]\\
		Y\ar@{-->}[rr]^\Phi&&X^\prime \subset\Bbb P^{9},}$$
	where $\widetilde{Y} \to X^\prime \subset\Bbb P^{9}$ is given by the linear system $|(i-1)\rho^\ast H-E|$, and $$\varphi_{|(i-1)\rho^\ast H-E|}:\xymatrix{\widetilde{Y} \ar[r]^{\varphi} & X\ar[r]& X^\prime \subset\Bbb P^{9}}$$
	is the Stein factorization.
		\end{rem}
	
\begin{lem} \label{mukai}	The variety $X^\prime$  is a Mukai fourfold with at worst terminal Gorenstein singularities and rank of  ${\mathrm{Pic}}\,(X^\prime)$ is $1$.
\end{lem}
	\begin{proof}
		Since $\mu$ is a divisorial Mori contraction, $X^\prime$ has at worst terminal singularities. It follows from ${\mathrm{rank}}\ {\mathrm{Pic}}\, (\widetilde{Y} )= 2$ that we have ${\mathrm{rank}}\ {\mathrm{Pic}}\, (X^\prime) = 1$. Since
		$$-K_{\widetilde{Y}} = i\rho^\ast H  - E=2((i-1)\rho^\ast H -E)-D,$$
		we get that $-K_{X^\prime}=2L$. Hence $-K_{X^\prime}$ is an ample Cartier divisor divisible by $2$ in ${\mathrm{Pic}}\, (X^\prime)$. So $X^\prime$ is a Mukai fourfold.
	\end{proof}
	
	\begin{rem}\label{rem4} \rm
	
The morphism $X\to X^\prime \subset \Bbb P^{9}$ is given by the linear system $|L| = | -\frac{1}{2} K_{X}|$. By \cite[Lemma 2.5]{HHT22}, this is an isomorphism. In the sequel, we identify $X$ with $X^\prime $ and $\varphi_{|(i-1)\rho^\ast H-E|}$ with $\varphi$.
Hence, $\varphi$ is birational, ${\rm d}_X=12$, and the morphism $\varphi$ contracts the divisor $D$ to an irreducible surface $\Sigma \subset X$.

Since ${\mathrm{rank}}\ {\mathrm{Pic}}\, (\widetilde{Y})= 2$, the exceptional locus of $\varphi$ coincides with $D$, and
$D$ is a prime divisor. Therefore, $\varphi$ has at most a finite number of $2$-dimensional fibers. By the Theorem in \cite{AnW98}, $\Sigma$ has at most isolated singularities. Moreover,  the surface $\Sigma$ as in Table \ref{S95} is normal. 

	\end{rem}

\begin{lem}\label{blowup} 
The morphism $\varphi : \widetilde{Y}\to X$ is the blowup of the surface $\Sigma$, where both $\Sigma$ and $X$ are smooth if surface $F$ is general.

\end{lem}
\begin{proof}
Assume that $\Sigma$ or $X$ are singular, the extremal $K_{\widetilde{Y}} $-negative contraction $\varphi : \widetilde{Y} \to X$  would have
a $2$-dimensional fiber, say, $\widetilde{Z}\subset D\subset \widetilde{Y}$. Since $\Sigma$ is normal, by the main theorem and Proposition 4.11 in \cite{AnW98} we have $\widetilde{Z}\cong\Bbb P^2$ and
$$(i\rho^\ast H-E)|_{\widetilde{Z}}=-K_{\widetilde{Y}}|_{\widetilde{Z}}=\mathcal O_{\Bbb P^2}(1).$$
Since $\widetilde{Z}$ is contracted to a point under $\varphi$, we have $((i-1)\rho^\ast H-E)|_{\widetilde{Z}}\sim 0$.
Thus, $\rho^\ast H|_{\widetilde{Z}}=\mathcal O_{\Bbb P^2}(1)$ and $E|_{\widetilde{Z}} =\mathcal O_{\Bbb P^2}(i-1)$. Therefore, the image $Z=\rho (\widetilde{Z})\subset \rho(D)$  is 
a plane meeting $F$ along a curve of degree $i-1$, i.e, $Z\neq \Sigma$ and $Z\cap F\cong \widetilde{Z}\cap E$ is a curve of degree $i-1$ in $Z\cong \Bbb P^2$, which contradicts with a general surface $F$ contain in the unique hypersurface $W$ of degree $i-2$ of $Y$ and $W$ contains finite planes.
Hence,  all the fibers of $\varphi$ have dimension less than or equal to $1$. 
By \cite{And85},  both $X$ and $\Sigma$ are smooth, and $\varphi$ is the blowup of $\Sigma$. 
\end{proof}

Finally, the proof of Theorem \ref{invpro} can be proceed as follows.
	\begin{proof}[Proof of Theorem \ref{invpro}]
	To $(1)$ and $(2)$: By Lemma \ref{fano2}, $\widetilde{Y}$ is a Fano fourfold with rank of ${\rm Pic}(\widetilde{Y})$ is $2$. It follows from Lemma \ref{mori2} that there exists a Mori contraction $\mu: \widetilde{Y} \longrightarrow X^\prime$ different from $\rho.$ By Lemma \ref{noample2}, $(i-1)\rho^\ast H - E$ is not ample, so it yields a supporting linear function of the extremal ray generated by the curves in the fibers of $\rho$.  By Lemma \ref{birational2}, $\mu$ is birational and its exceptional locus coincides with $D.$ It implies from Lemma \ref{mukai} and Remark \ref{rem4} that $X^\prime$ is a Mukai fourfold of genus $7$. The birationality of $\Phi$ and the commutativity of the diagram follows.
	
	To $(3)$: It follows from Lemma \ref{blowup}.
	
	To $(4)$: It implies from Lemma \ref{noample2} that $D \sim (i-2)\rho^\ast H-E.$ From Lemma \ref{birational2} and Remark \ref{rem4}, we get $(i-1)\rho^\ast H-E=\varphi^\ast L$. The two remain relations follows.
\end{proof}

\section{Flexibility of affine cones over Mukai fourfolds of genus $7$}
In this section, we explore the flexibility of the affine cone over a Mukai fourfold of genus $7$. A point $p\in V$ in an affine algebraic variety $V$ is called \textit{flexible} if the tangent space $T_p V$ is spanned by tangent vectors to the orbits of actions of the additive group of the field $\mathbb{G}_a$ on $V$. If every smooth point of $V$ is flexible, we call $V$ a \textit{flexible variety}. One effective way to prove the flexibility of the affine cone over a projective variety is to construct a cylindrical subset, as proposed by Kishimoto, Prokhorov, and Zaidenberg (\cite{KPZ11}). This powerful technique has been used to establish the flexibility of many significant classes of varieties, including the affine cone over a Mukai fourfold of genus $7$. In this section, we use this approach to demonstrate the flexibility of the affine cone over such a fourfold.

\begin{defn} 
\rm
	\begin{enumerate}[$(1)$]
\item   An {\it $\mathbb{A}^n$-cylinder} in $X$ is defined as a pair $(Z, \varphi)$ where $Z$ is a variety, $\mathbb{A}^n$ is the affine $n$-space over $\mathbb{C}$, and $\varphi: Z\times \mathbb{A}^n\to X$ is an open embedding. We define $X$ to be {\it $\mathbb{A}^n$-cylindrical} if there exists an $\mathbb{A}^n$-cylinder $(Z, \varphi)$ in $X$. If $H$ is a divisor on $X$, we say that an $\mathbb{A}^n$-cylinder $(Z, \varphi)$ in $X$ is \emph{$H$-polar} if $\varphi(Z\times\mathbb{A}^n)=X\setminus\textrm{supp }D$ for some effective divisor $D\in|kH|$, where $k>0$. (\cite[Definitions 3.1.5, 3.1.7]{KPZ11}).

		\item 
		A subset $Y\subset X$ is called \emph{invariant} with respect to a $\Bbb A^n$-cylinder $(Z, \varphi)$ in $X$ if $$Y\cap \varphi( Z\times \Bbb A^n)=\varphi(\pi_1^{-1}(\pi_1(\varphi^{-1}(Y)))),$$ where $\pi_1:Z\times \Bbb A^n\to Z$ is the first projection of the direct product (\cite[Definitions 3]{Per13}). 
		\item
		We say that a variety $X$ is {\it transversally covered} by $\Bbb A^n$-cylinders ${\{(Z_i, \varphi_i)\}_{i\in I}}$ in $X$ if 
		$X$ has a  covering 
		\begin{equation}\label{covering}
			X=\bigcup\limits_{i\in I} U_i
			\tag{$*$}
		\end{equation}
		where each $U_i$ is a Zariski open subset in $X$ such that $U_i= \varphi_i(Z_i\times \Bbb A^n)$ and it does not admit any proper invariant subset with respect to  every $\Bbb A^n$-cylinders $(Z_i, \varphi_i)$. 
				 (\cite[Definitions 4]{Per13})
	\end{enumerate}
\end{defn}
In the following, we will mention some results that are useful for our work. For more information and proof, refer to \cite{HT22}.
\begin{lem}[\cite{HT22}]\label{cylinderCremona}
Let $W$ be a singular cubic hypersurface in $\Bbb P^4$.
Then the variety $ \Bbb P^4\setminus W$ is transversally covered by $\Bbb A^2$-cylinders ${\{(Z_{y}, \varphi_{y})\}_{y\in  \Bbb P^4\setminus W }}$ in $ \Bbb P^4\setminus W$.

\end{lem}

\begin{lem}\label{openCover}
For each point $x\in X$, there exists an open set $U_x$ of $X$ such that 
\begin{enumerate}[$(1)$]
\item $U_x=X\backslash C_x$, where $C_x$ is a cubic hypersurface section of $X$.
\item There exists  a singular cubic threefold $W_x\subset\Bbb P^4$,  and a birational map $\Phi_x:X \dashrightarrow  \Bbb P^{4}$ such that  $$X\backslash C_x\cong \Bbb P^4\backslash W_x.$$
\end{enumerate}
\end{lem}
\begin{proof}
To prove our result, we analyze the intersection of the Fano-Mukai fourfold $X$ of genus $7$ with its tangent space and tangent cone at each point. 
For every point $x\in X$, the intersection can take one of four forms, as shown in \cite{LM03} and \cite{Kuz18}: a set of lines, or containing a plane, quadratic surface, or cubic scroll surface.

The first case, where the intersection is a set of lines, has been previously studied in \cite{HT22}.

In the second case, where the intersection contains a plane, we can establish the result using Theorem \ref{pro1} and \cite[Theorem 3.1]{PrZ16}. By Proposition \ref{pro1}, we can find a hyperplane $H_x$ in $\mathbb{P}^9$ and a hyperplane $H'_x$ in $\mathbb{P}^6$ such that $X\setminus H_x$ is isomorphic to $X_{2\cdot 2}\setminus H'_x$. Then, using \cite[Theorem 3.1]{PrZ16}, we can show that the complement of a quadric hypersurface section $Q$ of a complete intersection of two quadrics $X_{2\cdot 2}\subset \Bbb P^6$ is isomorphic to the complement of a quadric hypersurface $Q'_x$ in $\Bbb P^4$. Hence, there exists a birational map $\Phi_x:X \dashrightarrow \Bbb P^{4}$ such that $X\backslash C_x\cong \Bbb P^4\backslash W_x$, where the hypersurface $C_x$ is the union of a hyperplane $H_x$ and a quadric hypersurface $Q_x$, while the hypersurface $W_x$ is the union of a hyperplane $H_x'$ and a quadric hypersurface $Q_x'$.

In the third case, where the intersection contains a quadric surface, we can use Proposition \ref{pro1} to find a quadric hypersurface section $Q_x$ of $X$ and a quadric hypersurface section $Q'_x$ of a quadric hypersurface $Y\subset \mathbb{P}^5$ such that the complement of $Q_x$ in $X$ is isomorphic to the complement of $Q'_x$ in $Y$. Then, we can observe that the complement of a hyperplane section $H$ of $Y\subset \Bbb P^5$ is isomorphic to the complement of a hyperplane $H'_x$ in $\Bbb P^4$. As a result, we obtain a birational map $\Phi_x:X \dashrightarrow \Bbb P^{4}$ such that $X\backslash C_x\cong \Bbb P^4\backslash W_x$, where the hypersurface $C_x$ is the union of a quadric hypersurface $Q_x$ and a hyperplane $H_x$, while the hypersurface $W_x$ is the union of a quadric hypersurface $Q_x'$ and a hyperplane $H_x'$.

Finally, for the case where the intersection contains a cubic scroll surface, we can use Theorem \ref{pro1} to  prove our result.
\end{proof}

Let us remind ourselves of the criteria for determining flexibility in affine cones
\begin{thm}[\cite{Per13}]\label{flexibilityCriterionPer}
If for some very ample divisor $H$ on a smooth projective variety $X$ there exists a transversal covering by $H$-polar $\Bbb A^1$-cylinders,
then the affine cone over $X$ is flexible. 
\end{thm}
\begin{lem}[\cite{HT22}]\label{flexibilityCriterion} If $X$ is {\it transversally covered} by $\Bbb A^2$-cylinders, then $X$ is {\it transversally covered} by $\Bbb A^1$-cylinders.
In particular,  if for some very ample divisor $H$ on a smooth projective variety $X$ there exists a transversal covering by $H$-polar $\Bbb A^2$-cylinders,
then the affine cone over $X$ is flexible. 

\end{lem}

\begin{thm}\label{general}
The affine cones over Fano-Mukai fourfold of genus  $7$ are flexible.
\end{thm}
\begin{proof}
Let $X$ be a Fano-Mukai fourfold of genus  $7$. Since $x\in C_x$ for all $x\in X$ and$\bigcap\limits_{x\in X}C_x=\emptyset$ the family $\{U_{x}\mid x\in X \}$ is an open covering of $X$. Therefore the variety $X$ is   transversally covered by $\Bbb A^2$-cylinders using Lemmas \ref{cylinderCremona} and \ref{openCover}. Specifically, for each point $x\in X$, we have an open set $U_x$ such that $U_x=X\backslash C_x$, where $C_x$ is a cubic hypersurface section of $X$, and a singular cubic threefold $W_x\subset\Bbb P^4$, and a birational map $\Phi_x:X \dashrightarrow \Bbb P^{4}$ such that $X\backslash C_x\cong \Bbb P^4\backslash W_x$. Lemma \ref{cylinderCremona} tells us that $ \Bbb P^4\setminus W_x$ is transversally covered by $\Bbb A^2$-cylinders, and so we can use the birational map $\Phi_x$ to pull back this covering to $U_x\subset X$. Finally, we can glue these coverings together to obtain a global  transversally $H$-polar $\Bbb A^2$-cylinder cover of $X$. 
By Lemma \ref{flexibilityCriterion}, we know that if a smooth projective variety $X$ is transversally covered by $H$-polar $\Bbb A^2$-cylinders, then $X$ is also transversally covered by $H$-polar $\Bbb A^1$-cylinders. Then by Theorem \ref{flexibilityCriterionPer}, the affine cone over $X$ is flexible.
\end{proof}


\begin{bibdiv}
	\begin{biblist}
		\bib{And85}{article}{
      author={Ando, T.},
       title={On extremal rays of the higher-dimensional varieties},
        date={1985},
     journal={Invent. Math.},
      volume={81},
      number={2},
       pages={347\ndash 357},
}

		\bib{AFKKZ13}{article}{
			author={Arzhantsev, I.},
			author={Flenner, H.},
			author={Kaliman, S.},
			author={Kutzschebauch, F.},
			author={Zaidenberg, M.},
			title={Flexible varieties and automorphism groups},
			date={2013},
			journal={Duke Math. J.},
			volume={162},
			pages={767\ndash 823},
		}
	\bib{AZK12}{article}{
		author={Arzhantsev, I.},
		author={Zaidenberg, M.},
		author={Kuyumzhiyan, K.},
		title={{Flag varieties, toric varieties, and suspessions: three examples
				of infinite transitivity}},
		date={2012},
		journal={Mat. Sb.},
		volume={203},
		pages={3\ndash 30},
	}
	
\bib{AnW98}{article}{
	author={Andreatta, M.},
	author={Wi\'{s}niewski, J.A.},
	title={On contractions of smooth varieties},
	date={1998},
	journal={J. Algebraic Geom.},
	volume={7},
	number={2},
	pages={253\ndash 312},
}

\bib{AnM03}{article}{
	author={Andreatta, Marco},
	author={Mella, Massimiliano},
	title={Morphisms of projective varieties from the viewpoint of minimal
		model theory},
	date={2003},
	journal={Dissertationes Mathematicae},
	volume={413},
	pages={1\ndash 72},
}

\bib{Arz23}{unpublished}{
	author={Arzhantsev, I.},
%	author={Zaidenberg, M.},
	title={Automorphisms of algebraic varieties and infinite transitivity},
	date={2023},
	note={preprint: \url{https://arxiv.org/pdf/2212.13616.pdf}},
}

%\bib{CPPZ21}{article}{
%	author={Cheltsov, I.},
%	author={Park, J.},
%	author={Prokhorov, Y.},
%	author={Zaidenberg, M.},
%	title={Cylinders in Fano varieties},
%	date={2021},
%	journal={EMS Surv. Math. Sci.},
%	volume={8},
%	pages={39\ndash 105},
%}
%
%\bib{FKZ17}{article}{
%	author={Flenner, F.},
%	author={Kaliman, S.},
%	author={Zaidenbeg, M.},
%	title={Cancellation for surfaces revisited},
%	date={2017},
%	journal={arXiv:1801.02274 (2017), 26p},
%	%volume={413},
%	%pages={1\ndash 72},
%}
%
%\bib{For06}{article}{
%	author={ Forstner{i}c, F.},
%	%author={Kaliman, S.},
%	%author={Zaidenbeg, M.},
%	title={Holomorphic flexibility properties of complex manifolds},
%	date={2006},
%	journal={Amer. J. Math.},
%	volume={128},
%	pages={239\ndash 270},
%}
%


\bib{HHT22}{article}{
	author={Hang, N. T.~A.},
	author={Hoff, M.},
	author={Truong, H.L.},
	title={On cylindrical smooth rational fano fourfolds},
	date={2022},
	journal={J. Korean Math. Soc.},
	volume={59},
	pages={87\ndash 103},
}

\bib{HT22}{article}{
	author={Hoff, M.},
	author={Truong, H. L.},
	date = {2022},
	title={ Flexibility of affine cones over Mukai fourfolds of genus $g \geq 7$},
	journal ={https://doi.org/10.48550/arXiv.2208.09109,}
	
}



%
%\bib{IsP99}{book}{
%	author={Iskovskikh, V.~A.},
%	author={Prokhorov, Yu.},
%	title={Fano varieties. {Algebraic geometry V.}},
%	publisher={Encyclopaedia Math. Sci. Springer},
%	address={Berlin},
%	date={1999},
%	volume={47},
%}

%\bib{Kap18}{article}{
%	author={{Kapustka}, Micha{\l}},
%	title={{Projections of Mukai varieties}},
%	language={English},
%	date={2018},
%	ISSN={0025-5521; 1903-1807/e},
%	journal={{Math. Scand.}},
%	volume={123},
%	number={2},
%	pages={191\ndash 219},
%}
%
%\bib{KZ99}{article}{
%	author={{Kaliman}, SH.},
%	author={Zaidenbeg, M.}
%	title={{Affine modifications and affine hypersurfaces with very transitive automorphism group}},
%	date={1999},
%	%ISSN={0025-5521; 1903-1807/e},
%	journal={{Transformation Groups}},
%	volume={4},
%	number={1},
%	pages={53\ndash 95},
%}

\bib{KPZ11}{book}{
	author={Kishimoto, T.},
	author={Prokhorov, Yu.},
	author={Zaidenberg, M.},
	editor={Daigle, D.},
	editor={Ganong, R.},
	editor={Koras, M.},
	title={Group actions on affine cones},
	series={Affine Algebraic Geometry CRM Proceedings and Lecture Notes},
	publisher={American Mathematical Society},
	address={Providence},
	date={2011},
	volume={54},
}

\bib{Kuz18}{article}{
	author={{Kuznesov}, A.},
	title={{On linear section of the spinor tenfold, I}},
	%language={English},
	date={2018},
	ISSN={0025-5521; 1903-1807/e},
	journal={{Izv. Math.}},
	volume={82},
	number={694},
	%pages={191\ndash 219},
}

\bib{LM03}{article}{
	author={{Landsberg}, Joseph~M.},
	author={{Manivel}, Laurent},
	title={{On the projective geometry of rational homogeneous varieties}},
	language={English},
	date={2003},
	ISSN={0010-2571; 1420-8946/e},
	journal={{Comment. Math. Helv.}},
	volume={78},
	number={1},
	pages={65\ndash 100},
}

%\bib{macaulay2}{misc}{
%	author={Grayson, D.~R.},
%	author={Stillman, M.~E.},
%	title={{\sc Macaulay2} --- {A} software system for research in algebraic
%		geometry (version 1.18)},
%	date={2021},
%	note={{H}ome page: \url{http://www.math.uiuc.edu/Macaulay2/}},
%}

\bib{Mat02}{book}{
	author={Matsuki, K. },
	%author={Mukai, S.},
	title={Introduction to 	the Mori Program},
	date={2002},
	%series={Affine Algebraic Geometry CRM Proceedings and Lecture Notes},
	publisher={ Springer-Verlag New York},
	%address={Providence},
%	volume={36},
%	number={2},
%	pages={147\ndash 162},
}
%
%\bib{MoM81}{article}{
%	author={Mori, S.},
%	author={Mukai, S.},
%	title={Classification of Fano $3$-folds with $b_2 \geq 2$},
%	date={1981},
%	journal={Manuscr. Math.},
%	volume={36},
%	number={2},
%	pages={147\ndash 162},
%}
%
%\bib{Muk89}{article}{
%	author={Mukai, S.},
%	title={Biregular classi cation of {F}ano {$3$}-folds and {F}ano
%		manifolds of coindex {$3$}},
%	date={1989},
%	journal={Proc.Nat.Acad. Sci. U.S.A.},
%	volume={86},
%	number={9},
%	pages={3000\ndash 3002},
%}
%	
	\bib{MPS18}{article}{
		author={Michalek, M.},
		author={Perepechko, A.},
		author={S\"{u}ss, H.},
		title={Flexible affine cones and flexible coverings},
		date={2018},
		journal={Math. Z.},
		volume={290},
		pages={1457\ndash 1478},
	}
	
	\bib{Per13}{article}{
		author={Perepechko, A. Yu.},
		%author={Zaidenberg, M.},
		title={Flexibility of affine cones over del Pezzo surfaces of degree 4 and 5,},
		date={2013},
		journal={Funct. Anal. Appl.},
		volume={47},
		number={4},
		pages={284\ndash 289},
	}

\bib{PaW16}{article}{
	author={Park, J.},
	author={Won, J.},
	title={Flexible affine cones over del Pezzo surfaces of degree 4,},
	date={2016},
	journal={Eur. J. Math.},
	volume={2},
	number={1}
	pages={304\ndash 318},
}

%\bib{PaW16}{article}{
%	author={Park, Jihun},
%	author={Won, Joonyeong},
%	title={{Flexible affine cones over del Pezzo surfaces of degree {$4$}}},
%	date={2016},
%	journal={European Journal of Mathematics},
%	volume={2},
%	number={1},
%	pages={304\ndash 318},
%	url={https://doi.org/10.1007/s40879-015-0054-4},
%}

%\bib{PCS05}{article}{
%	author={Przhiyalkovski, V. V.},
%	author={Cheltsov, I. A.},
%	author={Shramov},
%	title={Hyperellipti and trigonal Fano threefolds}, 
%	date={2005},
%	journal={Russian).Izv. Ross. Akad. Nauk Ser. Mat.},
%	volume={69},
%	pages={145\ndash 204},
%	journal={English translation in Izv. Math.},
%	date = {2005}
%	volume={69},
%	pages={365\ndash 421},
%}



\bib{PrZ16}{article}{
	author={Prokhorov, Y.},
	author={Zaidenberg, M.},
	title={Examples of cylindrical {F}ano fourfolds},
	date={2016},
	journal={Eur. J. Math.},
	volume={2},
	pages={262\ndash 282},
}
%		
%	\bib{PrZ17}{article}{
%		author={Prokhorov, Y.},
%		author={Zaidenberg, M.},
%		title={New examples of cylindrical {F}ano fourfolds,},
%		date={2017},
%		journal={Adv. Stud. Pure Math.},
%		volume={75},
%		pages={443\ndash 463},
%	}

\bib{PrZ20}{unpublished}{
	author={Prokhorov, Y.},
	author={Zaidenberg, M.},
	title={Affine cones over {F}ano-{M}ukai fourfolds of genus 10 are
		flexible},
	date={2020},
	note={preprint: \url{https://arxiv.org/abs/2005.12092}},
}
	
	
%	
%\bib{Sub14}{article}{
%	author={S\"{u}B, Y.},
%	%author={Zaidenberg, M.},
%	title={Fano threefolds with 2-torus action—a picture book},
%	date={2014},
%	journal={Doc. Math.},
%	volume={19},
%	pages={905\ndash 914},
%}
%	
	
\end{biblist}
\end{bibdiv}
\end{document}